\documentclass[11pt,dvipsnames,reqno]{amsart}

\newcommand{\calF}{\mathcal{F}}

\newcommand{\calH}{\mathcal{H}}

\newcommand{\calL}{\mathcal{L}}

\newcommand{\calO}{\mathcal{O}}

\newcommand{\frakS}{\mathfrak{S}}
\newcommand{\frakA}{\mathfrak{A}}

\newcommand{\bbZ}{\mathbb{Z}}
\newcommand{\bbR}{\mathbb{R}}
\newcommand{\bbF}{\mathbb{F}}
\newcommand{\bbC}{\mathbb{C}}
\newcommand{\bbP}{\mathbb{P}}
\newcommand{\bbG}{\mathbb{G}}
\newcommand{\bbQ}{\mathbb{Q}}
\newcommand{\bfF}{\mathbf{F}}

\newcommand{\la}{\langle}
\newcommand{\ra}{\rangle}
\newcommand{\cha}{\textup{char}}

\newcommand{\SL}{\textrm{SL}}
\newcommand{\Aut}{\text{Aut}}

\newcommand{\GL}{\text{GL}}

\newcommand{\Spec}{\text{Spec}}

\newcommand{\adj}{\text{adj}}
\newcommand{\PSL}{\text{PSL}}

\newcommand{\SU}{\mathrm{SU}}
\newcommand{\PSU}{\mathrm{PSU}}
\newcommand{\da}{\dasharrow}

\newcommand{\Sp}{\mathrm{Sp}}
\newcommand{\PSp}{\mathrm{PSp}}
\newcommand{\PO}{\mathrm{PO}}

\newcommand{\half}{\frac{1}{2}}
\newcommand{\Cr}{\mathrm{Cr}}
\newcommand{\cre}{\mathrm{crd}}
\newcommand{\Bir}{\mathrm{Bir}}

\newcommand{\TP}{\mathrm{P}}
\newcommand{\Or}{\mathrm{O}}

\newcommand{\rc}{\mathrm{rcd}}

\newcommand{\ed}{\mathrm{ed}}
\newcommand{\rank}{\mathrm{rank}}

\newcommand{\PGL}{\mathrm{PGL}}

\newcommand{\beq}{\begin{equation}}
\newcommand{\eeq}{\end{equation}}

\usepackage{graphicx}
\usepackage{amscd, amsmath, amsthm, amssymb,xy, xypic}
\usepackage{url}
\xyoption{arrow}

\xyoption{matrix}
\input xy
\xyoption{all}
\newtheorem{theorem}{Theorem}[section]
\newtheorem{lemma}[theorem]{Lemma}
\newtheorem{proposition}[theorem]{Proposition}

\theoremstyle{definition}     

\newtheorem{example}[theorem]{Example}

\theoremstyle{remark}

\numberwithin{equation}{section}

\title{The essential and Cremona dimensions of a group}

\author{Igor Dolgachev}
\address{Department of Mathematics, University of Michigan, 525 East University Avenue, Ann Arbor, MI 48109-1109 USA}
\email{idolga@umich.edu}

\begin{document}
\begin{abstract}
The Cremona dimension of a group $G$ is the minimal $n$ such that $G$ is isomorphic to a subgroup of the Cremona group 
of birational transformations of an $n$-dimensional rational variety. In this survey article, we give many examples 
that give evidence to the conjecture that the Cremona dimension of a finite group over the field of complex 
numbers is less than or equal to 
the essential dimension 
of the group. 
\end{abstract}

\maketitle

\section{Introduction}
The notion of the essential dimension of a finite group $G$ was introduced by A. Buhler and Z. Reichstein in 
1997 \cite{Buhler}. Since then, it has become the subject of investigation in different fields of algebra 
and algebraic geometry. Roughly speaking, the essential dimension $\ed_k(G)$ of $G$ over a fixed field $k$ 
is the smallest number of parameters needed to define a generic equation over $k$ with the Galois group isomorphic to $G$. 
Viewing the Galois extension defined by the equation as a torsor of $G$, this notion can be extended 
to any algebraic group over $k$. 

In a talk at the Joint Meeting of the American and Chilean Mathematical Societies in Puc\`on in 2010, based on some 
examples in small dimension, the author conjectured that the Cremona dimension of a finite group over 
the field of complex numbers is less than or equal to the essential dimension of the group. In the present paper, 
we give a more systematic 
list of examples confirming this conjecture. In particular, we check the conjecture for some finite $p$-groups, for  
groups generated by pseudo-reflections, and some simple or almost simple groups. 
I use this opportunity to survey some known results on the essential dimension and subgroups of the Cremona group.

The paper is an extended version of my talk at the CIRM Conference on the Cremona group, held in Luminy in May 2025. 
I thank many participants of the conference 
for their helpful discussions. 

We work over an algebraically closed field $k$ of characteristic $p \ge 0$. We employ the ATLAS notations for 
finite groups (see \cite{Atlas}). 
For example, if $p$ is prime, $p^n$ denotes the elementary abelian group $(\bbZ/p\bbZ)^{\oplus n}$, 
the special linear groups $\SL_n(\bbF_q)$ over a finite field of cardinality $q = p^s$ are 
denoted by $\SL_n(q)$, and the usually simple groups $\PSL_n(\bbF_q)$ are denoted by $L_n(q)$. As it is customary, 
we denote the alternating and symmetric groups by $\frakA_n$ and $\frakS_n$. We denote the orthogonal 
group of $\bbF_q^n$ equipped with a nondegenerate quadratic form by $\Or^{\pm}(n,\bbF_q)$, where 
the two types differ by the value of the Witt index.  The Atlas 
notation for this group is $\mathrm{GO}^{\pm}_{2n}(q)$. The usually simple group obtained from the quotient of a subgroup 
of $\mathrm{SO}^{\pm}(n,\bbF_q)$ of index 1 or 2 by the center is denoted by $\Or^{\pm}_n(q)$. They coincide 
with the commutator subgroups of $\mathrm{SO}^{\pm}(2n,\bbF_q)$. 

The 
semi-direct product $A : B$ of two groups is denoted by $A : B$.

\vskip7pt
I thank Serge Cantat, Alex Duncan, and Zinovy Reichstein for their help in writing some parts of this paper.
I owe much to the referee for the substantial improvement in the paper's exposition.

\section{The Essential and Cremona dimensions: definitions} \label{S:2}
The origin of the notion of the essential dimension $\ed_k(G)$ of a finite group $G$ goes 
back to the classical notion of 
the Tschirnhaus transformation (Ehrenfried Walther von Tschirnhaus, 1659–1708). 
The formal definition was provided by Buhler and Reichstein in 1997 \cite{Buhler}.

Let $L/K$ be a finite separable extension of fields containing a fixed subfield $k$. The essential dimension 
$\ed_k(L/K)$ of $L/K$ is the smallest transcendence degree of a subfield $E/k$ of $K$ such that 
$L$ is generated by an element with the minimal polynomial  $g(t)\in E[t]$.  
 
Let $G$ be a finite group. The essential dimension $\ed_k(G)$ of $G$ is the essential dimension of the field 
extension
$k(V)/k(V)^G$, where $V$ is a linear faithful representation of $G$ over $k$. It can be shown that $\ed_k(G)$ 
does not depend on the choice of $V$.
It is easy to see that $\ed_k(G)$ is equal to the minimal dimension of a $G$-variety $X$ (we assume that the 
action of $G$ is \emph{birational and faithful}) such 
that there exists an equivariant rational map $V\dasharrow X$ (a \emph{compression}).

For example, we can take $K = k(a_1,\ldots,a_n)$, where 
$a_1,\ldots,a_n$ are algebraically independent  over a field $k$ of characteristic $0$, and consider 
$L = K[t]/(f(t))$ with minimal polynomial 
$f(t) = t^n+a_1t^{n-1}+\cdots+a_n$. We set
$d_k(n):= \ed_k(L/K)$. Since the Galois group of $f(t)$ is isomorphic to $\frakS_n$ which 
 acts linearly in $\overline{K}^n$ by permuting the roots, we see that 
$$d_k(n) = \ed_k(\frakS_n).$$
 By Tschirnhaus's transformation (also known as Tschirnhausen or Tschirnhaus transformation) that replaces $t$ 
 with $t+\frac{1}{n}a_1$, we get $d_\bbC(n)\le n-1$. It is easy to 
 see 
that $d_\bbC(3)=1$ and  $d_\bbC(4) = 2$. Hermite proved that $d_\bbC(5) \le 2$ and Klein proved that $d_\bbC(5) = 2$. 
Serre proved that 
$d_\bbC(6) = 3$ \cite[3.6]{Serre} and Duncan proved that $d_\bbC(7) = 4$ \cite{Duncan2}. It is known that, for $n\ge 5$, 
$$n-3\ge \ed_\bbC(\frakS_n) = d_k(n) \ge \Bigl[\frac{n}{2}\Bigr] $$
(see \cite{Buhler}). Moreover, the function $d_k(n)$ is not decreasing  and
\begin{eqnarray}\label{alternate}
\ed_k(\frakS_{n+2})\ge \ed_k(\frakS_n)+1
\end{eqnarray}
\cite[6.3]{Buhler}.  It is conjectured that $\ed_k(\frakS_n) = n-3$ if $\cha(k) = 0$ 
\cite{ReichsteinSym}. 

A compression 
$X$ of $V$ is obviously 
\emph{$G$-unirational}, i.e., it is the image of a dominant rational $G$-map  $\bbP^n\to X$. In particular, $X$ is a 
unirational; hence, it is
a  \emph{rationally connected} algebraic $G$-variety. 
We define the \emph{rational connected dimension} of $G$ to be
$$\rc_k(G) = \min\{\dim X: \textrm{$X$ is a rationally connected $G$-variety over $k$}\}.$$
Obviously, 
\begin{eqnarray}\label{eq:2.2}
\rc_k(G)\le \ed_k(G).
\end{eqnarray}
 
A rational variety is rationally connected. We can introduce the \emph{Cremona dimension}
 $$\cre_k(G) = \min\{n: \ G\  \textrm{embeds in $\Cr_k(n)$}\}.$$
Here, $\Cr_k(n)$ denotes the group of birational automorphisms of $\bbP_k^n$. In algebraic terms, 
$\Cr_k(n) = \Aut(k(t_1,\ldots,t_n))/k)$. Obviously,
$$\rc_k(G) \le \cre_k(G).$$
In a talk at the Joint Meeting of the American and Chilean Mathematical Societies in Puc\`on in 2010, I conjectured that
for $k = \bbC$, 
\begin{eqnarray}\label{myconjecture}
\cre_k(G)\le \ed_k(G).
\end{eqnarray}
Since I have no idea how to prove this conjecture, I suggest looking for a 
counterexample. However, it will be hard since many examples are discussed in this paper 
suggest that the conjecture could be true.

\vskip6pt
A finite group of essential dimension $0$ is trivial. A finite group $G$ with $\ed_\bbC(G) = 1$ is isomorphic to a 
cyclic group or a dihedral group $\mathfrak{D}_{4k+2}$ of order $2(2k+1)$. In fact, among all subgroups of $\PGL_2(\bbC)$, 
only these groups can be isomorphically lifted to a subgroup of $\GL_2(\bbC)$.

A group of essential dimension $2$ over a field of characteristic $0$ must be isomorphic to a subgroup of $\Cr_k(2)$. In the case $k  = \bbC$, groups of 
essential dimension $2$  were classified by 
Duncan \cite{Duncan}. There are only a few of them among the possible subgroups of $\Cr_k(2)$. They are 
subgroups of $\GL_2(\bbC), \ L_2(7)$,\ $\frakS_5$, or  $\mathbb{G}_m^2:H$, 
where $H =\mathfrak{D}_{2n}, n = 3,4,6$.

\section{Basic properties}\label{S:3}
It is obvious that 
\beq\label{obvious}
\cre_k(H) \le \cre_k(G)
\eeq
if $H$ is a subgroup of $G$, and 
\beq\label{product}
\cre_k(G_1\times G_2) \le \cre_k(G_1)+\cre_k(G_2).
\eeq
The same properties are true for the essential dimension \cite[Lemma 4.1]{Buhler}. 

Less well known is the behavior of the Cremona and essential dimensions under group extensions.
 If $Z = \bbZ/p\bbZ$ and $Z.G$ is a central cyclic extension of $G$, then it is known that 
 \begin{eqnarray}\label{center}
 \ed_k(G) = \ed_k(G/Z)+1
\end{eqnarray}
 if $Z\cap [G,G] = \{1\}$ and $p\ne \cha(k)$ \cite[Theorem 5.3]{Buhler}. 
 
 If $p = \cha(k)$, and 
 $$1\to G_1\to G \to G_2\to 1$$
 is an extension with an abelian $p$-group $G_1$, then 
\beq\label{vistoli}
\ed_k(G) \le \ed_k(G_1)+\ed_k(G_2)
\eeq
 (see \cite{Ledet}, \cite[Corollary 4.6]{Tossici}). 
 
 On the other hand, for any normal subgroup $N$ of $G$, we, obviously,  have 
 $$\cre_k(G) \ge \rc_k(G/N).$$

 One introduces the notion of the essential dimension $\ed_k(G,p)$ at a prime $p$ 
  (see \cite{Karpenko}). One of the equivalent definitions is $\ed_k(G,p): = \ed_k(\textrm{Syl}_p(G))$, where 
  $\textrm{Syl}_p(G)$ is any Sylow $p$-group of $G$. It follows that, for any $p$, 
  \beq\label{peddim}
  \ed_k(G)\ge \ed_k(G,p).
  \eeq 
    The computation of $\ed_k(G,p)$ is easier than the computation of 
  $\ed_k(G)$. In fact, if $\cha(k)\ne p$, the essential dimension of a finite $p$-group $P$ is equal to the 
  minimal dimension of a faithful linear 
  representation of $P$ \cite[Theorem 4.1]{Karpenko}. This implies  
  \beq\label{pgroups}
  \cre_k(P) \le \ed_k(P),
  \eeq
  because $P$ acts faithfully on the affine space of dimension $\ed_k(P)$. 
  
  In the case where $G$ is an abelian group of rank $n$ and the exponent of $G$ is coprime to $\cha(k)$, 
  \beq\label{abelian}
  \ed_k(G) = \rank(G)
  \eeq
  \cite[Theorem 5.1]{Buhler}.

  \section{Fixed points of an action of an abelian group}\label{S:4}
Suppose $\cha(k) = 0$ and $G$ is an abelian group. It follows from \eqref{eq:2.2} and \eqref{abelian} that $\rc_k(G) \le \rank(G)$. 
The following proposition is a nice result of Koll\`ar and Zhuang \cite{Kollar} that gives a lower bound for any 
abelian $p$-group $G$.

\beq\label{kollar}
\rc_k(G) \ge \frac{p-1}{p}\rank(G).
\eeq
This follows from the following proposition:

\begin{proposition} Assume $\cha(k) = 0$. Let $G$ be an abelian $p$-group of rank 
$r$ acting on a smooth connected projective variety $X$ of dimension $n$. Then,
$$r \le \nu_p(\chi(X,\calO_X))+\frac{p}{p-1} n.$$
\end{proposition}

\begin{proof} First, we claim that $G$ contains a subgroup $H$ of index $p^c$ such that 
$$c \le \nu_p(\chi(X,\calO_X))+\frac{n}{p-1}.$$
We skip the details of the proof. It is based on the equivariant theory of 
Chern classes. If $H$ is a subgroup such that $[G:H] = p^{c}$ is the smallest orbit of the action of $G$, 
then all $G$-invariant algebraic cycles are of degree divisible by $p^{c}$. It follows from Hirzebruch's formula 
that the largest $p$-power dividing the denominator $D$
of the Todd class $\mathrm{Td}_n(X)$ is equal to $[\frac{n}{p-1}]$. Thus,
$\chi(X,\calO_X) = \mathrm{Td}_n(X) = (D\mathrm{Td}_n(X))/D$ and  
$\nu_p(\chi(X,\calO_X))\ge c-[\frac{n}{p-1}]$.

This yields
$$r\le  \nu_p(\chi(X,\calO_X))+\frac{p}{p-1}n.$$
In fact, $H$  is an abelian subgroup that acts on $T_x(X)$. 
Since we assumed that $\cha(k) = 0$, it acts faithfully on $T_x(X)$. Hence, the rank of $H$ does not exceed $n$. So, 
$$r \le c+n\le \nu_p(\chi(X,\calO_X))+\frac{n}{p-1}+n \le  
\nu_p(\chi(X,\calO_X))+\frac{p}{p-1}n.$$
\end{proof}

The inequality \eqref{kollar} follows from the facts that $\chi(X,\calO_X) = 1$ for a rationally connected variety 
and that a birational action of a finite group on a complex algebraic variety $X$ can be isomorphically lifted 
to a biregular action on a nonsingular model of $X$. The proof depends on 
the existence of an equivariant resolution of singularities (true in dimension $2$ for any field).

The following example is taken from \cite{Prokhorov}.

\begin{example} Assume $k = \bbC$.
Let $G = p^r$ with $\rc_k(G)\le 3$. Then,
$$r \le \begin{cases}
      7& \text{if}\  p = 2, \\
      5& \text{if} \  p = 3, \\
      4& \text{if} \  p = 5,7,11,13 \\
      3& \text{if} \  p\ge  17 .
\end{cases}
$$
The inequality \eqref{kollar} improves Prokhorov's bound by one if $p\le 13$.
\end{example}

\begin{example} Assume $\cha(k)\ne 2$. We have 
$$\cre_k(2^{r}) \le \Bigl[\frac{r}{2}\Bigr]$$
since $2^{2s}\cong \mathfrak{D}_4^s$ acts faithfully on $(\bbP^1)^s$. The inequality \eqref{kollar}
shows that 
\begin{equation}\label{kollar2}
\cre_k(2^{r}) = \rc_k(2^{r}) < \ed_k(2^{r})  = r.
\end{equation}
\vskip5pt
\end{example}

\section{Extraspecial $p$-groups} \label{S:5}

 A finite $p$-group $G$ is called \emph{extraspecial} if its center $Z(G)$ is cyclic and coincides with the commutator 
 subgroup $[G,G]$. As a consequence, $Z(G)$ is of order $p$ and $\bar{G} = G/Z(G) \cong p^{2n}$ (see 
 \cite[Chapter 5, 5.5]{Gorenstein}).

An extraspecial $p$-group can be generated by  elements $x_1,\ldots,x_n,y_1,\ldots,y_n$ and a central 
element $c$ of order $p$ with  relations 
$$[x_i,y_j] = 1,\  [x_i,y_i] = c, \ c^p = 1, \ x_i^p\in \la c\ra,\  y_i^p \in \la c\ra.$$ 
If $p$ is odd, there are two isomorphism classes, one of exponent $p$ denoted by $p_+^{1+2n}$ and one of exponent 
$p^2$ denoted by $p_-^{1+2n}$. In the former case, $x_i^p =y_i^p = 1$. In the latter case,
we may take $x_1^p = c, x_i^p =1, i\ne 1, y_i^p = 1$.

Extraspecial $2$-groups have exponent $4$, however, there are still two isomorphism classes. 
We may take $x_1^2 = y_1^2 = c, x_i^p = y_i^p = 1, i\ne 1$, or 
$x_1^2 = y_i^2 = 1$ for all $i$. In the former case, the group is denoted by $2_+^{1+2n}$, and in the 
latter case it is denoted by 
$2_-^{1+2n}$. The groups differ by the isomorphism class of the quadratic form 
$q:G/Z(G)\cong \bbF_2^{2n}\to Z(G)\cong  \bbF_2$ induced by the map of $G$ defined by $ x\mapsto x^2$. This is a 
quadratic form on $G/Z$ associated with the 
symplectic form induced by the commutator $[x,y]$. It is even in the former case, and odd 
in the latter case.  

The splitting field $K$ of the group $G = p_\pm^{1+2}$ for linear representations over a field of characteristic $0$ 
is equal to $\bbQ(\omega)$, where $\omega$ is a primitive $p$-root (resp. $p^2$-root) of unity if 
$G\cong p_+^{1+2}$ (resp. $G\cong p_+^{1+2}$). It admits $p^2+p-1$ non-isomorphic irreducible linear representations 
over $K$. Among them 
are $p-1$ faithful representations of dimension $p$ and $p^2$ one-dimensional representations trivial on the center.
The faithful representations $\phi:G\to \GL_p(K)$ defined by sending 
$x$ to the cyclic permutation matrix, and sending $y$ to the diagonal matrix $[\omega^{i(p-1)},\ldots,\omega^i,1]$ if $G\cong p_+^{1+2}$
and $[\omega^{i(1+(p-1)p)},\ldots,\omega^{i(1+p)}, \omega^i], 1\le i\le p-1,$
if $G\cong p_-^{1+2}$ (see \cite[Chapter 5, 5.5]{Gorenstein}).

The subgroup $H$ generated by $x_1,y_1$ is an extraspecial group $p_{\pm}^{1+2}$ of order $p^3$. The group 
$p_+^{1+2}$ is isomorphic to the group of upper triangular unipotent $3\times 3$-matrices with entries in 
the finite field $\bbF_p$. The group $p_-^{1+2}$ is isomorphic to the semi-direct product 
$p^2 : p$, where a generator 
of the quotient group acts on a generator $x$ of the normal subgroup by $x\mapsto x^{1+p}$.

The centralizer subgroup $C_G(H)$ 
is isomorphic to $p_+^{1+2(n-1)}$. By induction, we obtain that $G$ is isomorphic to the central product 
of $n$ copies of 
extraspecial $p$-groups $p_+^{1+2}$ or the central product of one extraspecial group $p_-^{1+2}$ and $n-1$ copies 
of the extraspecial $p$-groups $p_+^{1+2(n-1)}$. In the case $p =2$, $2_+^{1+2}\cong \mathfrak{D}_8$ is the dihedral group of order $8$, and we obtain 
that $2_+^{1+2n}$ is isomorphic to the central product of $n$ copies of $\mathfrak{D}_8$. Also, 
$2_-^{1+2}\cong Q_8$, the quaternion group, and we obtain 
that $2_-^{1+2n}$ is isomorphic to the central product of 
$Q_8$ and $n-1$ copies of $\mathfrak{D}_8$.

An irreducible linear representation of a central product of groups $G_i$ with the same splitting field $K$ is isomorphic 
to the tensor product of irreducible representations of the groups $G_i$ over $K$ 
\cite[Chapter 3, Theorem 7.1]{Gorenstein}. This shows that the group $p_{\pm }^{1+2n}$ admits $p-1$ non-isomorphic 
irreducible linear representations of dimension $p^n$ and $p^{2n}$ one-dimensional representations.

Recall that  $G/Z(G)\cong \bbF_p^{2n}$ has a structure of a symplectic linear space over $\bbF_p$. 
The linear subspaces generated by the cosets $x_i$ and the cosets 
of $y_i$ form a pair of complementary maximal isotropic subspaces. It follows that 
the quotient of the normalizer $N(G)$ of the subgroup $G$ in $\SL_{2n}(p)$ modulo $G$ is a subgroup $H$ of $\Sp_{2n}(p)$.
This defines an extension 
$$1\to p_{\pm}^{1+2n}\to H(p,n) \to H\to 1.
$$
It is known that 
\beq\label{winter}
H = \begin{cases}
      \Sp_{2n}(p)& \text{if $G =  p_+^{1+2n}, p > 2$}, \\
      p_-^{1+2(n-1)}:\Sp_{2n-2}(p)&\  \text{if $G =  p_-^{1+2n}, p > 2$}, \\
      \Or^{\pm}(2n,\bbF_2)\  \text{if $G =  2_\pm^{1+2n}$}
\end{cases}
\eeq
 (see \cite{Winter}).

Suppose an extraspecial $p$-group $G = p^{1+2n}$ acts on a complex rational variety $X$ of dimension $n$. 
Then, the quotient $Y = X/Z(G) $ is unirational and, hence, rationally connected. The inequality \eqref{kollar} yields  
$$p^n\ge \cre_\bbC(G)\ge \rc_\bbC(p^{2n}) \ge \frac{p-1}{p}2n.$$

\vskip6pt
Following Zinovy Reichstein (a personal communication), we will prove the following:

\begin{proposition}\label{mainresult}
\beq\label{reich2}
\cre_\bbC(p_{\pm}^{1+2n}) \le pn.
\eeq
\end{proposition}

\begin{proof}
 As we stated before, the group
$G = p_{\pm}^{1+2n}$ is the central product of the $p$-extraspecial subgroups $G_i$ of order $p^3$. 
Recall that the central product $A\circ B$ of two groups with a fixed isomorphism 
$\theta:Z_1\to Z_2$ of central subgroups of $A$ and $B$ is the quotient of the direct product $A\times B$ by the subgrpup $Z_1\times Z_2$ of elements 
$(z_1,\theta(z_1)^{-1})$. More generally, let $Z_1,\ldots,Z_n$ be central subgroups of groups $A_1,\ldots,A_n$ with 
fixed isomorphisms $\theta_i:Z_i\to Z$ to some group $Z$. The central product 
$A_1\circ \cdots \circ A_n$ 
is the quotient of $A_1\times\cdots \times A_n$ by the  subgroup of $Z_1\times\cdots\times Z_n$ of elements of the form
$(z_1,\ldots,z_n)$, where $\theta_1(z_1)\cdots \theta_n(z_n) = 1$.

 In our case, $G$ is the central product $G_1\circ\cdots\circ G_n$ of extraspecial $p$-groups $G_i$
 order $p^3$, where we identify 
 the centers with a cyclic group 
 of order $p$ isomorphic to 
 $\mu_p \subset \bbC^*$. Thus, $G \cong (G_1\times \cdots G_n)/Z$, where 
 $Z$ is the subgroup of $\mu_p^n\subset (\bbC^*)^n$ of elements $(z_1,\ldots,z_n)$ with $z_1\cdots z_n = 1$. 
 The group $G$ acts faithfully 
 on the quotient $(\bbC^p)^n/Z$. Since the quotient of a linear action by an 
abelian group 
is known to be a rational variety, we conclude the proof.
\end{proof}
\vskip6pt

It follows from the result of Karpenko and Merkurjev, which we cited in Section \ref{S:3}, that 
$$\ed_\bbC(p^{1+2n}) = p^n.$$
So, if $n > 1$,  \eqref{mainresult} gives a strong inequality
$$\cre_k(p^{1+2n}) < \ed_\bbC(p^{1+2n}).$$
Note that, in our case, one cannot apply \eqref{center} 
 because the commutator subgroup of $p^{1+2g}$ coincides with its center.

\vskip6pt
The group $p_+^{1+2n}$ is known in the theory of abelian varieties as a subgroup $\mu_p.\bbF_p^{2n}$ of 
the \emph{Heisenberg group} $\calH_{p,n}$, a central extension $\bbC^*.\bbF_p^{2n}$.  We identify the linear 
space $V = \bbF_p^n$  with 
the group algebra $\bbF_p[P]$ of the subgroup $P$ generated by 
$x_1,\ldots,x_n$ modulo the center. Fix an isomorphism $e:Z(G) = \la c\ra\to \mu_p$ that sends $c$ to 
 a  primitive $p$-root $\omega$ of unity, and let $G$ act on complex valued functions $f(t), t\in P,$  by
$$(x_i\cdot f)(t)  = f(t+x_i), \quad (y_i\cdot f)(t) = e([y_i,t])f(t), \quad (c\cdot f)(t) = \omega f(t).$$
 The choice of an isomorphism $\bbZ/p\bbZ\to \mu_p$ defines $p-1$ non-isomorphic representations
of the extra-special group coming from the \emph{Schr\"odinger representation} of $\calH_{p,n}$, where $\bbC^*$ 
acts by scalar multiplication.

\begin{example} Assume $p = 3$ and $n = 1$. The Heisenberg group $3_+^{1+3}$ has a linear representation in 
$\bbC^3$. The normalizer group $N$ of $G$ in $\Sp_2(\bbC)\cong \SL_2(\bbC)$ is isomorphic to the extension 
\eqref{winter}
$$1\to 3_+^{1+3}\to H_{3,1} \to \SL_2(\bbF_3)\to 1.$$
The quotient group $H_{3,1}/Z(H_{3,1})$ is isomorphic to a subgroup $G_{216}$ of order $216$ of $\PGL_3(\bbC)$, 
classically known as the \emph{Hessian group}. 
The group $G_{216}$ leaves invariant the 
Hesse pencil of cubic curves. The subgroup $3^2 = G/Z(G)$ leaves invariant all members of the pencil, and  
the quotient group $\SL_2(3)$ acts on the pencil via the natural action of 
$\SL_2(3) \cong 2.\frakA_4$ in $\bbP^1$.
The group $3_+^{1+3}$ admits a faithful action on the triple cyclic cover of $\bbP^2$ branched over any member of 
the Hesse pencil.
This is a cyclic cubic surface, so $\cre_\bbC(3_+^{1+3}) = 2$ improving \eqref{mainresult} in this case.
\end{example}

\begin{example}
 Assume $n = 2$ and $p = 3$. The group $3_+^{1+4}$ acts faithfully and linearly in $\bbC^9$. 
 For any principally polarized abelian variety $(A,\calL)$ of dimension $2$, the linear system 
 $|\calL^{\otimes 3}|$ embeds $A$ in $\bbP(\Gamma(A,\calL^{\otimes 3}))\cong \bbP^8$. The group of $3$-torsion 
 points $A[3]\cong 3^4$ acts on $A$ but does not lift to an action on $\Gamma(A,\calL^{\otimes 3})^\vee$. However, 
 its central 
 extension $3_+^{1+4}$ lifts to a linear action on this space and realizes the Schr\"odinger representation in $\bbC^9$.
 It was proved by Arthur Coble that the image of $A$ in $\bbP^8$ is equal to the singular locus of a $G$-invariant cubic 
 hypersurface, 
 called the \emph{Coble cubic}. This remarkable fact has been extended to all $n = \dim A \ge 2$ 
 and $p = 3$ in \cite{BeauvilleCoble}. 
 
 The group $G$ acts faithfully on 
 the affine cone 
 over the Coble cubic in $\bbC^{3^n}$. Since a singular cubic hypersurface is rational, we obtain 
 that $\cre_\bbC(3_+^{1+2n})\le 3^{n}-1$. Reichstein's inequality \eqref{reich2} improves substantially on this inequality.
 
 The group $\Sp_{2n}(3) = H_{3,n}/3_+^{1+2n}$ acts in $V: = \Gamma(A,\calL^{\otimes 3})) \cong \bbC^{3^n}$.
 The corresponding linear representation is not irreducible, but decomposes into the direct sum of two linear  subspaces 
 $V^+$ and $V^-$ of dimensions $\half (3^{n}+1)$ and $\half (3^n-1)$, respectively. 
 They are eigensubspaces of the central involution of $\Sp_{2n}(3)$. The action of $\Sp_{2n}(3)$ (resp. $\PSp_{2n}(3)$) in 
 $V^-$ (resp. $V^+$) is the faithful linear representation of the smallest dimension.

 In the special case  $n = 2$, the group $\PSp_4(3)$ is a simple group isomorphic to the index 2 subgroup $W(E_6)'$ 
 of the Weyl group $W(E_6)$ of the root lattice of type $E_6$.
 In its projective representation in $\bbP(V^+)\cong \bbP^4$, it leaves invariant a rational quartic hypersurface. 
 It is, classically, known as the 
 \emph{B\"urhardt quartic}. It is distinguished from other quartic hypersurfaces in $\bbP^4$ by the 
 property that it admits 
 a maximal possible number of isolated singular points 
 (equal to $45$). This shows that 
 \beq
 \cre_\bbC(\PSp_4(3)) = 3,\quad \cre_k(\Sp_4(3)) =  4.
 \eeq 
 (the classification of finite subgroups of $\Cr_C(2)$ shows that $\cre_\bbC(\PSp_4(3)) > 2$, and, since $\Sp_4(3)$ is quasisimple,  we will see later that 
 $\rc_k(\Sp_4(3)) > 3$).
 
 We will also see later that $\ed_\bbC(\PSp_4(3)) \ge  4$, and since $\PSp_4(3)$ acts faithfully on 
 $V^+ \cong \bbC^5$, we obtain $\ed_\bbC(\PSp_4(3))\le 5$. In any case, 
 $$\cre_k(\PSp_4(3)) \le \ed_\bbC(\PSp_4(3)).$$
 
 It is known that, for any prime $p> 2$,
 $\ed_\bbC(\Sp_{2n}(p^r),p)\ge rp^{n-1}$ \cite{Knight}. In particular, since we know that $\Sp_4(3)$ acts faithfully 
 on $V^-\cong \bbC^4$, we obtain 
 $4\ge \ed_\bbC(\Sp_4(3)) \ge 3$. Therefore, if $\ed_\bbC(\Sp_4(3)) = 3$, we get a counter-example to our conjecture.
 Note that the group $\Sp_4(3)$ is a perfect group, so we cannot apply \eqref{center} to obtain that 
 $\ed_\bbC(\Sp_4(3)) = 4$.
 
The group $2\times \PSp_4(2)$ acts 
in $V^+ \cong \bbC^5$ as a group generated by complex pseudo-reflections. It is the group  
 No. 33 in the Shephard-Todd list of irreducible complex pseudo-reflection groups.

 The group $\Sp_4(3)$ acts in $V^-\cong \bbC^4$ and defines a finite subgroup of $\Aut(\bbP^-) \cong \PGL_4(\bbC)$.
 The group $3\times \Sp_4(3)$ acts in $V_2(3)^-$ as a group generated by complex pseudo-reflections. 
 It is group No. 32 in the Shephard-Todd list. It was first discovered by H. Maschke and, classically known, as the 
Maschke group.

\end{example}

A curious remark is that $G = 2_+^{1+24}$ is a subgroup of the Monster Group $\mathbf{M}$, which occurs
 as the centralizer of an involution. The quotient $N_\mathbf{M}(G)/G$ is isomorphic to the subgroup of 
 $\Or_{24}^+(2)$ isomorphic to the group $2.\textrm{Co}_1$, where $\textrm{Co}_1$ is a simple sporadic group, 
 the first Conway simple group.
Thus, $\ed_k(\mathbf{M}) \ge \ed_k(2_+^{1+24}) = 2^{12}$. It is known that the minimal 
dimension of a faithful linear 
representation 
of $\mathbf{M}$ is equal to $196383$. Therefore, $2^{12} \le \ed_k(\mathbf{M})\le 196383$. 

\section{Pseudo-reflection groups} \label{S:6}
Recall that a non-trivial linear transformation of finite order in a linear space $V$ over a field $K$ 
is called a \emph{pseudo-reflection} if 
it is the identity on some hyperplane in $V$. A pseudo-reflection group is a finite group $G$ generated by a 
finite set 
of pseudo-reflections.
It is a finite reflection group if it can be generated by a finite set of reflections, i.e., pseudo-reflections of order $2$. 
The dimension of 
the linear space $V$ is called 
the \emph{degree} of $G$ and is denoted by $d(G)$. If $K$ is algebraically closed and the order of $G$ is not divisible
by $\cha(K)$, the algebra of invariant polynomials $K[V]^G$ is freely generated by homogeneous polynomials of degrees 
$2\le d_1\le \cdots\le d_n$, where $n = \dim V$. The product $d_1\cdots d_n$ is equal to the order of $G$. We will 
assume that $G$ is a \emph{tame} pseudo-reflection group, i.e.,
its order is prime to $\cha(k)$.

If $k$ is algebraically closed, which we continue to assume,  a pseudo-reflection $g$ has $1$ as its 
eigenvalue of multiplicity $\dim V-1$ and one 
simple eigenvalue equal to a primitive $n$-root of unity, where $n$ is the order of $g$.

If $\cha(k) = 0$, the list of isomorphism classes of irreducible (i.e., not isomorphic to the product of 
pseudo-reflection groups) 
pseudo-reflection groups coincide with the Shephard-Todd list of complex unitary reflection groups. It consists of three 
infinite families of groups $\bbZ/n\bbZ, \  \frakS_n, \ G(m,s,n)$ and $34$ isolated cases, which are numbered 
by $4,\ldots, 37$.

 We already discussed the symmetric groups in Section \ref{S:2}. We recall that 
 $n-2\ge \ed_\bbC(\frakS_n)\ge 4$ if $n \ge 7$. The same result is known for $\rc_\bbC(\frakS_n)$ 
 \cite{ProkhorovSym}. The Cremona dimension $\cre_k(\frakS_n)$ is not known. It is easy to see that 
 $\cre_k(\frakS_n)\le n-3$, for example, considering its action on $\TP_1^n:= (\bbP^1)^n/\PGL(2)$.

The groups 
$G(m,s,n)$ are imprimitive groups isomorphic to a semi-direct product $A(m,s,n):\frakS_n$, where 
$A(m,s,n)$ is a subgroup of $(\bbZ/m\bbZ)^n$ of index $s$. The group $G(m,s,n)$ is 
generated by permutation matrices and diagonal matrices 
with $m$-th roots of unity on the diagonal and the determinant $d$ satisfying $d^{\frac{m}{s}} = 1$.
 Among these groups are some Coxeter reflection groups 
\[
\begin{split}
&G(1,1,n) = W(A_n),\quad  G(2,1,n) = W(B_n)\cong W(C_n),\\
& G(2,2,n) =  W(D_n),\quad  G(m,m,2) =  \textrm{I}_2(m).
\end{split}
\]
The groups Nos. 4--22 act in a $2$-dimensional linear space $V$.

\vskip5pt

The essential dimension of a tame pseudo-reflection group was studied by 
A. Duncan and Z. Reichstein \cite{DuncanReichstein}. They proved the following:

\begin{proposition}\label{az} Let $G$ be a tame irreducible pseudo-reflection group different from 
$\frakS_{n+1}\cong G(1,1,n)$. Then, 
$\ed_k(G(m,m,n)) = n-1$ if $(m,n)=1$, $\ed_k(W(E_6)) = 4$, and for all other groups $\ed_k(G) = d(G)$.
\end{proposition}

In particular, for all pseudo-reflection groups with $\ed_k(G) = d(G)$, we have 
$$\cre_k(G)\le \ed_k(G).$$

\vskip5pt
Let us explain the equality 

$$\ed_k(W(E_6)) = 4,$$
 which the authors attribute to me.

Let $W(E_n)$ be the Weyl group of a root system  of type $E_n$ defined by the Dynkin diagram of 
type $T_{2,3,n-3}$, and let 
\begin{eqnarray}\label{coble}
\rho_n: W(E_n)\to \Cr_k(2n-8)
\end{eqnarray}
be the Coble representation defined by the birational action of $W(E_n)$ on the rational quotient 
 $\TP_2^n:= (\bbP^2)^n/\textrm{PGL}_k(3)$ \cite{DO}.
In the case $n = 6$, the Coble representation is faithful, and we get
\beq
\cre_\bbC(W(E_6)) \le 4.
\eeq
 By putting six ordered points in $\bbP^2$ on an irreducible cuspidal cubic, one can define an equivariant rational map 
 $\mathbb{A}_k^6\da \TP_2^6$ that proves that $\ed_k(W(E_6))\le 4$. 
 
 It is proven in \cite{DuncanReichstein} that, for any tame pseudo-reflection group $G$, 
 $$\ed_k(G,p) = |\{i:p\vert d_i\}|.$$ 
 It follows from the classification that $p = 2$ or $3$ divides all $d_i$'s except in the case $G = W(E_6)$. 
 In the latter case, 
 $2$ divides four $d_i$. This shows that 
 $$4\ge \ed_k(W(E_6)) \ge \ed_k(W(E_6),2) = 4 = \cre_k(W(E_6)).$$ 
 
 \vskip6pt
 Let us compare the essential and Cremona dimensions of the pseudo-reflection groups. 
 
 We will use the following obvious lemma repeatedly.
 
 \begin{lemma}\label{obvious} Suppose a finite group $G$ acts linearly and faithfully on a linear space $V$. 
 Suppose it leaves invariant 
 a homogeneous polynomial $F$ such that the affine variety $V(F)$ is rational. Then, 
 $$\cre_k(G)\le \dim V-1,$$
 and the inequality is strict if the action descends to a faithful action on the projective space $\bbP(V)$.
 \end{lemma}

 A finite reflection group always admits a quadratic invariant defined by the symmetric Coxeter matrix (see 
 \cite[\S 2.3]{Humphreys}). It follows from the classification that the converse is also true: a pseudo-reflection group admiiting a quadratic 
 invariant is a reflection group. Applying the lemma and Proposition \ref{az}, we obtain 
 that, for any reflection group $G$, 
 not isomorphic to $\frakS_{n+1}$ or $W(E_6)$,
 $$\cre_k(G) \le \ed_k(G)-2\ \text{if $Z(G) = \{1\}$} $$
 and 
 $$\cre_k(G) \le \ed_k(G)-1\ \text{if $Z(G) \ne \{1\}$}.$$
The first case occurs only if $G$ is of type $D_{2t+1}$ or $\textrm{I}_{2t+1}$. 
In the latter case, $d(G) = 2$, so $\cre_k(G) = 1$ since the projective quadric left invariant by $G$ is the union
 of two points.

If $s \ne m$, $\ed_k(G(m,s,n)) = n$ and  $\ed_k(G(m,m,n)) = n-1$. In the former case, the essential dimension coincides with 
the smallest dimension of a faithful linear representation. So, $\cre_k(G(m,s,n))\le \ed_k(G(m,s,n))$. If $\frac{m}{s} \mid n$, the group
$G(m,s,n)$ leaves invariant the homogeneous polynomial 
$F = y^n+x_1\cdots x_n$ and acts faithfully on the corresponding affine variety. The permutation matrices act by permuting 
the unknowns $x_i$. Diagonal matrices multiply the variables $x_i$ and multiply $y$ by the determinant of the matrix.
The action descends to the action on the projective variety $X= V(F)\subset \bbP^n$ of dimension $n-1$ and degree $n$. Since the point 
$(0:\ldots:1)$ is a singular point of multiplicity $n-1$, the variety $X$ is rational. This gives 
$\cre_k(G(m,s,n))\le n-1$. So, if $s\ne m$, we get a strict inequality.

  \vskip6pt
 Using the classification of groups with $\cre_k(G)\le 2$ or $\ed_k(G)\le 2$, we obtain 
 that $\cre_k(G) = \ed_k(G)$ if $d(G) = 2,3$, unless $G$ is a reflection group of type $H_3$  (No. 23) isomorphic to 
 $2\times \frakA_5$. In this case,    $\ed_k(G) = 3$, but $\cre_k(G) = 2$ since the group acts faithfully on 
 $\bbP^1\times \bbP^1$.

 So, let us assume that $d(G) \ge 4$. They are the group numbers 28--37. 
 
  Group No. 28 is of degree $4$. It is the Coxeter reflection group of type $F_4$ of order 
  $1152$ isomorphic to the semi-direct product
  $W(D_4) : \frakS_3$  (see, for example,
   \cite[page 45]{Humphreys}, where it is also mentioned the realization of the group 
   as the group of symmetries of the $24$-cell due to Coxeter \cite[page 149]{Coxeter}). 
  The quotient of $W(F_4)$ by the central involution is realized as the group of symmetries 
  of a pencil of desmic quartic surfaces \cite{DolgKondo}. As such, it acts faithfully on 
  the universal family of the pencil
  $X\subset \bbP^3\times \bbP^1$. The projection $X\to \bbP^3$ makes $X$ a rational variety and gives
  $$\cre_k(W(F_4)/(\pm 1)) = \cre_k(W(F_4) = 3.$$

 The group No. 29 is of degree 4 and order $2^5\cdot 5!$. It is isomorphic to the group 
 $4\times G'$, where $G' \cong 2^4:\frakA_5$ is the group of projective transformations of $\bbP^3$ 
 leaving invariant the zero loci of one of the six irreducible factors of the fundamental invariant of degree 24 of the group No. 32. 
The six invariants define six quartic surfaces in $\bbP^3$ known 
as the \emph{Maschke quartic surfaces}. The group $G'$ is isomorphic to the upward double extension of the 
Mathieu group $M_{20} \cong 2^4:\frakA_5$  (see \cite[Remark 1]{DolgIco}, and the references therein). It is isomorphic 
to the stabilizer of a point or a line in the Mathieu group $M_{21}\cong L_3(4)$. It is also isomorphic to the index 2 subgroup 
of the Weyl group $W(D_5)$. We know that $\cre_k(G) \le \ed_k(G) = 4$ 
but I do not know whether $\cre_k(G) = 3$.

Group No. 30 is of degree 4. Its order is $14,400 = 120^2$. It is isomorphic to the Coxeter reflection 
group of type $H_4$. 
 Since the group has a quadratic invariant, $\cre_k(G) = 3 < \ed_k(G) = 4$. The group can be realized as a
 maximal 
 subgroup of the Weyl group $W(E_8)$. It is isomorphic to the double central extension of 
 $(\frakA_5\times \frakA_5) : 2$ (see \cite[Table 4.3]{Conway}). Note that it arises from a maximal subgroup
 of the simple group 
 of $\mathrm{O}^+_{8}(2) \cong W(E_8)'/(\pm 1)$. It has the same order but isomorphic to a different extension 
 $(\frakA_5\times \frakA_5) : (2^2)$ (see \cite[page 86]{Atlas}).

 Group No. 31 is of degree 4, and its order is $11,520 = 2^5\cdot 6!$. Note that the order of $G$ is equal to the order of 
 the Weyl group $W(D_6)$, but it is not isomorphic to this group. It is isomorphic to the group 
 $N_{2,2} = 2_+^{1+4}:\frakS_6$. Its extraspecial $2$-subgroup acts in $\bbP^3$, leaving invariant Kummer quartic 
 surfaces, the images of Jacobian varieties $A$ of curves of genus 2 under the map of degree $4$ given by the linear system $|2\Theta|$. 
 Its general orbits are the sets of 16 nodes
 of Kummer surfaces. Each involution in this group has two fixed lines that give rise to the \emph{Klein configuration} 
 $(30_6,60_3)$ of $30$ lines and their 60 
 intersection points. The quotient $\bbP^3/\calH_{2,2}$ is isomorphic to the Castelnuovo-Richmond quartic hypersurface
 with its Cremona-Richmond symmetri configuration $(15_3)$ of $15$ double lines intersecting by 3 at 15 points 
 (see \cite[10.2.1]{CAG}). The group $G$ leaves invariant the Klein configuration, the 60 points are the vertices 
 of  15 tetrahedra 
 that are permuted by $\frakS_6$. We mentioned above that the projectivized group No. 31 contains the projectivized 
 group No. 29 as a subgroup of index $6$
 
We already encountered groups No. 32 and No. 33 in the previous section. These are the 
 Burhardt and Maschke groups, 
which act linearly in the linear spaces $V_2(3)^{\pm}$. The Burhardt group $G$ leaves invariant 
the affine cone over a rational quartic hypersurface, and hence 
$\cre_k(G) = 4 < \ed_k(G) = 5$. 

Group No. 34 is of degree 6. Its order is $108\cdot 9!$, and it is isomorphic to the 
extension $6.\PSU_4(3).2 = 3.\SU_4(3).2$ of the simple group 
 $\PSU_4(\bbF_9)$ (in the ATLAS notation, the latter group is $\mathrm{U}_4(3)$). It is also isomorphic to 
 $\Or_6^-(3)$ and a subgroup 
 of $\GL_6(\bbC)$ that leaves invariant a lattice of rank $6$ over the Eisenstein numbers $\bbZ[\omega]$.
I do not know whether $\cre_\bbC(G) < 6$. 

Group No. 35 is the Weyl group $W(E_6)$. We discussed this group earlier. 

Group No. 36 is the Weyl group $ W(E_7)$.  The Coble representation $\rho_7$ is not faithful. It is the identity 
on the center
 generated by $w_0$.
We have $W(E_7) \cong W(E_7)'\times \la w_0 \ra$, where $W(E_7)'$ is a simple subgroup of index $2$ isomorphic to 
$\Sp_6(2)$. It is known that $\ed_k(\Sp_6(2),2) = 6$ \cite{Knight}. Thus,
$\ed_k(\Sp_6(2))\ge 6.$ The minimal dimension of an irreducible linear representatrion of $\Sp_6(2)$ is equal to $7$.
Since the restriction of $\rho_7$ to $W(E_7)'$ is faithful,  we obtain that
$$\cre_k(\Sp_6(2))\le 6\le \ed_k(\Sp_6(2)) \le 7.$$

Finally, the group No. 38 is the Weyl group $W(E_8)$. Since $W(E_8)$ admits a quadratic invariant,   
$\cre_k(G) < \ed_k(G) = 8$. 
The kernel of the Coble representation $\rho_8$ is generated by the central involution $w_0$. 
The group $W(E_8)/\la w_0\ra$ is isomorphic to $\Or^+(8,\bbF_2)$. 
We already mentioned that it contains a simple subgroup $W(E_8)'$ of index $2$ 
isomorphic to $\Or_8^+(2)$.
 It shows that 
 $$\cre_k(\Or_8^+(2)) \le 8.$$
It is known that $\ed_k(\Or_8^+(2),2) = 16$ \cite{Knight}. Thus,
$$\cre_k(\Or_8^+(2)) \le 8 < 16 \le \ed_k(\Or_8^+(2)).$$
Note that the minimal dimension of an irreducible linear representation of $W(E_8)'$ is equal to $28$. 

The extraspecial $2$-group $2_+^{1+8}$ is contained in $\Or_8^+(2)$. Applying the Coble representation $\rho_8$, 
we obtain that, in any characteristic,
$\cre_k(2_+^{1+8})\le 8$ confirming \eqref{mainresult}.

\section{Simple and quasi-simple groups}

 It follows from the classification of finite subgroups of $\Cr_\bbC(2)$ that the only simple groups contained 
 in $\Cr_\bbC(2)$ are $ \mathfrak{A}_5,\  \mathfrak{A}_6$, and $ L_2(7)$.

 Simple groups $G$ with $\cre_\bbC(G) = 3$ were classified by Prokhorov \cite{ProkhorovSimple}.
 They are 
 $$\frakA_7,\quad  L_2(8), \quad  W(E_6)' \cong \PSp_4(3)\cong \PSU_4(2).$$
  Any non-trivial action of a simple group is faithful. 
 The group $\frakA_7$ admits a linear $6$-dimensional representation. It acts in $\bbP^5$, leaving invariant 
  a smooth quadric. We identify the quadric with the Grassmannian $G_1(\bbP^3)$ and use  
  that the group of automorphisms of $G_1(\bbP^3)$ is isomorphic to a double extension of $\Aut(\bbP^3)$. 
  This implies that $\frakA_7$ acts by automorphisms of $\bbP^3$. 
  
  The group $L_2(8)$ admits a 9-dimensional irreducible representation $V$ and leaves invariant a smooth 
  quadric. It acts on the $10$-dimensional variety $\textrm{LG}(4,9) \subset \bbP^{15}$ of maximal isotropic 
  subspaces in $V$. 
  The projective space $\bbP^{15}$ contains two invariant subspaces of dimension $5$ and $7$. The intersection 
  of these subspaces with $\textrm{LG}(4,9)$ is a canonical curve of genus $7$ with the group of automorphisms of 
  the maximal possible order 
  $6\times 84 = 504 = |L_2(8)|$ and a rational Fano threefold of genus $7$, respectively.

 Finite simple groups with $\ed_\bbC(G) = 3$  were classified by Beauville \cite{Beauville}. It turns out that there 
 is only one group $G = \frakA_6$ 
 and, possibly, the group $G = L_2(11)$ that acts faithfully on the smooth Klein cubic hypersurface in $\bbP^4$ given by 
 the equation 
 $$x_0^2x_1+x_1^2x_2+x_2^2x_3+x_3^2x_4+x_4^2x_0 = 0.$$
 The inclusion of the latter group is conditional on the validity of a 
 conjecture of Cassels 
 and Swinnerton-Dyer: a cubic hypersurface over a field $K$ has a rational $K$-point if and only if it has a 
 $K$-rational 
 $0$-cycle of degree one, or, equivalently, when it has a rational point of degree prime to $3$ \cite{Coray}. 
Since, by Prokhorobv, $\cre_\bbC(L_2(11))\ge 4$, our conjecture and the conjecture 
 of Cassels and Swinnerton-Dyer are incompatible (this was noted earlier in \cite[Proposition 10.8]{DR2}).

  \vskip5pt
  Not much is known about simple groups $\frakA_n, n \ge 8$. If $\cha(k) = 0$, it is known 
  that $\ed_k(\frakA_n)$ is a non-decreasing function of $n$ and 
  $$\ed_k(\frakA_{n+4})\ge \ed_k(\frakA_n)+2.$$
  Also, it is known that 
  $$\ed_k(\frakA_n)\ge \begin{cases}
      \frac{n}{2}& \text{if $n$ is even}, \\
      \frac{n\pm 1}{2}& \text{if $n\equiv \mp 1 \mod 4$}
\end{cases} 
$$
\cite{Duncan}.  In particular,  $\ed_\bbC(\frakA_8) \ge 4$.  
The character table for the group $\frakA_8$ shows that the group admits an invariant quadric in its irreducible 
linear representation of dimension $6$, and hence, in its faithful projective representation in $\bbP^5$. This gives
$$\cre_k(\frakA_8) = 4 \le \ed_k(\frakA_8).$$

\vskip5pt
A finite group $G$ is called \emph{quasi-simple} if it is perfect and the quotient of $G$ by the center is a simple group.
A quasi-simple group $G$ with $\rc_k(G) = 3$ belongs to the 
following list:
$$\SL_2(7),\ 3.\frakA_6,  \ 2.\frakA_6 \cong \SL_2(9), \ 6.\frakA_6$$
 \cite{Blanc}. It is not known whether the last two groups can be realized. 

The group $3.\frakA_6$ admits an irreducible $3$-dimensional representation, so 
$$\cre_\bbC(3.\frakA_6) = \ed_\bbC(3.\frakA_6) = 3.$$ 

The group $G =\SL_2(7)$ admits an irreducible 
linear representation of dimension $4$. It acts faithfully on the Fano varieties isomorphic to 
the double covers  
of $\bbP^3$ ramified over the surfaces given by invariants of $G$ of degree $4$ or $6$.
Neither of these varieties is rational.  So, $\rc_\bbC( \SL_2(7)) = 3$, but I do not know 
 whether 
$\cre_\bbC(\SL_2(7))$ is equal to $3$.  We have $3\le \ed_\bbC(\SL_2(7))\le 4$. So, it can be a potential 
counterexample to my conjecture.

  Note that the group $2\times (3.\frakA_6)$ is not isomorphic to the group $\not\cong 6.\frakA_6$.  
It is a pseudo-reflection group No 27 with $d(G) = 3$, so 
 $$\cre_\bbC(2\times (3.\frakA_6)) = \ed_\bbC(2\times (3.\frakA_6)) = 3.$$
It is known that $2.\frakA_6$ admits a faithful linear 
 representation of dimension $4$, so 
 $\cre_k(2.\frakA_6) \le 4$.

\section{Positive characteristic}
Not much is known about the essential dimension of finite groups of order divisible by $\cha(k)$ (\emph{wild groups}). 
In this section, we assume that this is the case. As we will see, my conjecture is unlikely to be true in this case.

Since $\PGL_{n+1}(\bbF_{p^r})\subset \PGL_k(n+1)$ if $\cha(k) = p$, we obtain 
that 
$$\cre_k(\PGL_{n+1}(\bbF_{p^r})) \le n.$$
In particular,  $\cre_k(L_2(p^r)) = 1$.  Also, any group $p^r$ is contained in $k$, so 
$$\cre_k(p^r) = 1.$$

The classification of finite wild subgroups of $\PGL_2(k)$ and 
$\PGL_3(k)$ is known (see the lists and the references  in \cite{DolgMartinOdd}). We find the following groups 
with $\cre_k(G) = 1$ besides mentioned above:
$$
\mathfrak{D}_{4pk} (p\ne2),   \quad \mathfrak{D}_{4n+2} (p = 2), \quad A:\mu_s,\  A\subset k, (s,p) = 1.
$$ 
Some new groups of Cremona dimension $2$ appear as groups of automorphisms 
of del Pezzo surfaces in positive characteristic (see \cite{DolgMartinOdd}, \cite{DolgMartinEven}).
For example, the group $W(E_6)'$ is the group of automorphisms of the Fermat cubic surface in characteristic $2$ and 
the group $\PSU_3(2)$ is a group of automorphisms of a del Pezzo surface of degree $2$ in characteristic $3$.

Even for the cyclic group $\bbZ/p^r\bbZ$, we do not know either its essential dimension nor 
its Cremona dimension. It follows from \eqref{vistoli} 
 $$\ed_k(\bbZ/p^r\bbZ) \le r.$$
 It is conjectured in loc. cit. that the equality takes place. 
 
 Since any finite $p$-group is a nilpotent group, a more general fact follows from \eqref{vistoli}: for any finite 
 wild $p$-group $P$ of order $p^n$,
 \beq
 \ed_k(P)\le n.
 \eeq

 Following Alex Duncan (a private communication), let us prove the following:
 
 \begin{proposition}
 \beq
 \cre_k(\bbZ/p^2\bbZ) = \ed_k(\bbZ/p^2\bbZ) =  2.
 \footnote{This corrects a mistake in \cite[Theorem 8]{Dolgachev}, where it is asserted that $\cre_k(\bbZ/p^2\bbZ) > 2$ if $p > 2$.}
 \eeq 
 \end{proposition}
 
 \begin{proof}
 A generator $g$ of the group $G = \bbZ/p^2\bbZ$  acts regularly on the weighted homogeneous 
 plane $\bbP(1,1,p-1)$ and its minimal resolution $\bfF_2$ (considered as a smooth conic bundle over $\bbP^1$) 
 by the formula
\beq\label{oldaction}
g:(x,y,z)\mapsto (x+y,y, z-P(x,y)),
\eeq
 where $P(x,y) = x(x+y)\cdots (x+(p-2)y).$  We use that 
 $$\sum_{i=0}^{p-1}P(x+iy,y) = -y^{p-1}.$$
 This yields
 $$g^p:(x,y,z)\mapsto (x,y,z+y^{p-1}),$$
and implies that $g^p$ is of order $p$.
 \end{proof}
 
 In fact, as it was pointed out to me by Alex Duncan and 
 Gebhard Martin, one can use the translation action on the ring $W_{n}(k)$ of truncated Witt vectors of length $n$ to define 
 a faithful action of $\bbZ/p^n\bbZ$ on $\mathbb{A}_k^n\cong \Spec~W_n(k)$ proving
\beq\label{alexgebhard}
\cre_k(\bbZ/p^n\bbZ) \le n.
\eeq 
The translation action can be 
 extended to a smooth equivariant 
 compactification $\overline{W}_n(k)$ of $W_n(k)$   constructed inductively as a compactification 
 of a line bundle over $\overline{W}_{n-1}(k)$ starting from $\overline{W}_2(k) = 
 \bbP(\calO_{\bbP^1}\oplus \calO_{\bbP^1}(p))$ \cite{Garuti}.

 Recall that the ring $W_{n}(k)$ consits of vectors $a = (a_0,a_1,\ldots,a_{n-1})\in k^{n}$ with the addition 
 $a+b:= (S_0(a,b), S_1(a,b), \ldots, S_{n-1}(a,b)$, where the polynomials 
 $S_i(x,y)\in \bbZ[x_0,\ldots,x_i,y_0,\ldots,y_i]$
 are defined by the property 
 $$w_n(x_0,\ldots,x_{n})+w_n(y_0,\ldots,y_{n}) = 
 w_n(S_0(x_0,y_0),\ldots,S_{n}(x_0,\ldots,x_{n},y_0,\ldots,y_{n})	),$$
 where
 $$w_n(X_0,\ldots,X_{n}) = \sum_{i=0}^{n}p^iX_i^{p^{n}}.$$
 The multiplication map $[p]$ is given by $(a_0,\ldots,a_{n-1})\mapsto  (0,a_0^p,\ldots,a_{n-2}^p).$
  It follows that $[p]^{n}(W_{n}(k)) = \{0\}$. This implies that, if $a_0\ne 0$, the homomorphism  
 $\bbZ\to W_n(k), 1\mapsto a,$ defines an injective map 
 $$\iota_a: \bbZ/p^{n}\bbZ\to W_n(k),$$
 which is compatible with the multiplication by $p$ in the source and the target. Note that, taking 
 $a = (1,0,\ldots,0)$, we obtain the natural inclusion
 $$\bbZ/p^n\bbZ = W_n(\bbF_p) \subset W_n(k).$$
 
 This defines a faithful action of $\bbZ/p^{n}\bbZ$ on $W_n(k)\cong \mathbf{A}_k^{n}$ via the 
 translation by $a$, proving the inequality \eqref{alexgebhard}. 
 
 To get the explicit formula, we need to know the explicit formula for the polynomials $S_0,\ldots,S_{n-1}$. 
 Every exposition 
 of the theory of Witt vectors includes the formulas for $S_0$ and $S_1$:
 $$S_0(x_0,y_0) = x_0+y_0,\quad S_1(x_0,x_1,y_0,y_1) = x_1+y_1-\frac{(x_0+y_0)^p-x_0^p-x_1^p}{p}.$$
For example, taking $n = 2, p = 3, a = (1,0)$, we obtain that the action on the affine plane is defined by the formula
$(x,y) \mapsto (x+1, y-x^2-x).$
After homogenizing, we get the action on $\bbP(1,1,2)$ that coincides with the action from \eqref{oldaction}. 
Homogenizing differently,
we obtain the action $(x,y)\mapsto (x+t,y-x^2t-xt^2)$ on $\bbP(1,1,p)$. After we blow up the point $(0:0:1)$, we obtain 
the minimal ruled surface $\bbP(\calO_{\bbP^1}\oplus \calO_{\bbP^1}(3)) \cong \bfF_3$. This leads to the compactification from 
\cite{Garuti}

Let me give one more explicit example for an action of $\bbZ/p^3\bbZ$ in the case $p = 3$. Other examples
 with larger $n$ or larger $p$ use very large polynomials.

The general recursive formulas for the polynomials $S_i(x,y)$ were given by Matthieu Romagny \cite{Romagny}.
They give the following formula for $S_2$ in the case $p = 3$: 
\[
\begin{split}
&S_2(x_0,x_1,x_2,y_0,y_1,y_2) = \\
&-x_0^8y_0 - 4x_0^7y_0^2 - 9x_0^6y_0^3 - 13x_0^5y_0^4 - 13x_0^4y_0^5 - 9x_0^3y_0^6\\
& - 4x_0^2y_0^7 - x_0y_0^8 - x_0^4x_1y_0^2 - x_0^4y_0^2y_1 - 2x_0^3x_1y_0^2 - 2x_0^3y_0^2y_1\\
& + x_0^2x_1^2y_0 - x_0^2x_1y_0^2 + 2x_0^2x_1y_0y_1 - x_0^2y_0^2y_1 + x_0^2y_0y_1^2 + x_0x_1^2y_0\\
& + 2x_0x_1y_0y_1 + x_0y_0y_1^2 - x_1^2y_1 - x_1y_1^2 + x_2 + y_2.
\end{split}
 \]
 
 \vskip6pt
 Substituting $(y_0,y_1,y_2) = (1,0,0)$, we obtain the following action of $\bbZ/3^3\bbZ$ on $\mathbb{A}_k^3$:
 \[
 \begin{split}
  &(x,y,z) \mapsto (x+1,y-x^2-x,z-x-x^8 - 4x^7 - 9x^6 - 13x^5\\
  & -13x^4-9x^3-4x^2
- x^4y -2x^3y+ x^2y^2 - x^2y  + xy^2.
\end{split}
  \]
  We could homogenize the action in two ways. We can write 
  $$(t,x,y,z,t)\mapsto (t,x+t, y+f_2(x,t),z+f_8(x,y)),$$ 
  or 
 $$(t,x,y,z,t)\mapsto (t,x+t, y+f_3(x,t),z+f_9(x,y)),$$
 where the subscript denotes the degree of a homogeneous polynomial.
 The first projectivization gives the action on $\bbP(1,1,p-1,p^2-1)$, the second one defines the action 
 on $\bbP(1,1,p,p^2)$. The difference is the action on the boundary $V(t)$. In the first case, it 
 is not trivial, and in the second case, it is the trivial action. Lifted to $\overline{W}_n(k)$, the quotient by the group 
 action of $\bbZ/p^n\bbZ$ is a cyclic cover of $\overline{W}_n(k)$ ramified along the boundary 
 divisor \cite[Proposition 2]{Garuti}.

\vskip6pt
 A candidate for a counter-example to the inequality \eqref{myconjecture}  is the extraspecial $p$-group $p^{1+2n}$. 
 It follows from 
 \eqref{vistoli} that 
 $$\ed_k(p^{1+2n}) =  2.$$
 Although the Heisenberg group $p_+^{1+2}$ admits an embedding in $\Cr_k(2)$ via the action:
 $$\phi(x_1): (x,y,z) \mapsto (x+\alpha y,y+\beta  z,z),\quad \phi(y_1):(x,y,z) \mapsto (x+\gamma y,y+\delta  z,z),
$$
 where $\alpha\delta-\beta\gamma = 1$ (see \cite{Deserti}), I do not know how to embed any group $p^{1+2n}$ into $\Cr_k(2)$.
 
 The best I can show is that 
 $$\cre_k(p_+^{1+2n})\le 2n,$$
 where $p$ is odd. Fix an embedding $\bbF_p^{2n} =U\oplus V\hookrightarrow k$  with the standard symplectic form $\la x,y\ra$ for which each summand is a maximal isotropic subspace.  Consider the projective variety $X(F)\subset \bbP^{n}$, where
 $$F = x_0^p-x_0x_{2n+1}^{p-1}+\sum_{i=1}^nx_i(x_i^{p-1}+\alpha_i^{p-1}x_{2n+1}^{p-1})+\sum_{i=n+1}^{2n}x_i(x_i^{p-1}+\beta_i^{p-1}x_{2n+1}^{p-1}),$$
where $(\alpha_1,\ldots,\alpha_n)$ is a basis of $U$ and $(\beta_1,\ldots,\beta_n)$ is a basis of $V$. Consider the action of $p_+^{1+2n}$ on $X$ defined by 
\begin{eqnarray*}
c&:&(x_0,x_1,\ldots,x_{2n+1})\mapsto (x_0+x_{2n+1},x_1,\ldots,x_{2n+1}),\\
\alpha_i&:& \bigl(x_0,\ldots,x_{2n+1})\mapsto (x_0+\la \alpha_i,\beta_i\ra x_{2n+1},x_1,x_2+\alpha_ix_{n+1},\ldots,\\
&&x_n+(p-1)x_{2n+1}, x_{n+1},\ldots,x_{2n+1}\bigr),\\
\beta_i&: &\bigl(x_0,\ldots,x_{2n+1})\mapsto (x_0+\la \beta_i,\alpha_i\ra x_{2n+1},x_1,x_2,\ldots,x_n,\\
&&x_{n+1},x_{n+2}+\beta_ix_{2n+1},x_{2n}+(p-1)\beta_ix_{2n+1}\bigr),
\end{eqnarray*}
where we identify the center $Z$ of $p_+^{1+2n}$ with $\bbF_p,$ so that  $c=1\in Z$, and the quotient by the center is $U\oplus V$. The extension is given by the symplectic form on $U\oplus V$. The variety $X$ is singular along codimension two subvariety $V(\sum_{i=0}^{2n}x_i,x_{2n+1})$  with multiplicity $p-1$. This implies that $X$ is a rational variety.

\vskip6pt
Note that the group 
$2_+^{1+6}$ acts faithfully on a del Pezzo surface of degree one in characteristic $2$ \cite{DolgMartinEven}, so 
$\cre_k(2_+^{1+2n}) = 2$, for $n\le 3$.

\begin{example} Assume $p = \cha(k)\ne 2, 3$. Consider the Segre cubic hypersurface $\mathcal{S}(n)$ given 
by the equations in $\bbP^{n+1}$:
$$\sum_{i=0}^{n+1}x_i^3 = \sum_{i=0}^{n+1}x_i = 0.$$
The symmetric group $\frakS_{n+2}$ acts faithfully by permuting the coordinates. 
 If $n = 2k$ is even and $\cha(k)\ne 2$, the cubic hypersurface has 
$\binom{n+1}{[\frac{n}{2}]}$ nodes (it is known to be the maximal possible number of isolated singular points 
for a cubic hypersurface in $\bbP^{2k}$ in characteristic 0). 
As we mentioned earlier,  a singular cubic hypersurface is rational (via the projection from a singular point).
This confirms the fact that $\cre_k(\frakS_{n+2})\le n-1$. 
Assume that $p\mid (n+1)$. The point $P = (1:1:\ldots:1)$ is an additional 
singular point, which is fixed under the action of $\frakS_{n+2}$. The group acts faithfully 
on the tangent cone at $P$ of dimension $n-2$. This defines a faithful  action of $\frakS_{n+2}$ on a quadric 
of dimension $n-2$ and shows that 
\beq\label{a7charp}
 \cre_k(\frakS_{n+2})\le n-2
 \eeq
if $\cha(k)|(n+2), p\ne 2, 3$. The first possible case is $\frakS_{10}$ and $p = 5$. 

\end{example}

\begin{example} As we mentioned in Section 2, it is conjectured that $\ed_k(\frakS_n) \le  n-3$ if $\cha(k) = 0$. 
However, 
in \cite{ReichsteinSym}, the authors prove that $\ed_k(\frakS_n) \le n-4$ 
for certain special $n$ if $p = \cha(k)\mid n$. Their construction uses the action of the 
affine group $A = \mathrm{Aff}_1(k)$ on 
the affine $(n-2)$-dimensional quadric $Q^{n-2}$ given by equations
$$F_1 = \sum_{i=1}^{n}x_i^2= 0, \quad F_2 = \sum_{i=1}^{n}x_i= 0.$$
We assume that $n\ge 5, p \ne 2$, and $p\mid n$. 
The action is  
$$(x_1,\ldots,x_{n})\mapsto (ax_1+bx_{n},\ldots,ax_n+bx_{n}).$$
We have $F_1 \mapsto aF_1+bnx_{n}$ and $F_2\mapsto a^2F_2+2abF_1+b^2nx_{n}^2.$
Since $p\mid n$, the quadric is $A$-invariant. The action is obviously faithful and commutes with 
$\frakS_n$. It descends to the faithful 
action of $\frakS_{n}$ on the orbit space of dimension $n-4$. Since $x_1 = 0$ is a cross-section of the action with the stabilizer 
subgroup of a general point equal to $\bbG_m$, we find that the rational quotient $Q^{n-2}/A$ is birationally 
isomorphic to a quadric of dimension $n-4$ in $\bbP^{n-3}$. In particular, it is a rational variety. This gives
$$\cre_k(\frakS_n)\le n-4, \ n\ge 5.$$
For example, if  $n =  p = 5$, 
$\cre_k(\frakS_5) = 1$. The group $\frakS_5$ is isomorphic to $\PGL_2(\bbF_5)$. If $n = 6, p = 3$, we get $\cre_k(\frakS_6) \le 2$. In fact,  
$\frakS_6\cong \PGL_2(9)\cong \Sp_4(\bbF_2)$ \cite[page 4]{Atlas}, and  $\cre_k(\frakS_6) = 1$.

To compare this with the essential dimension, we use \cite{ReichsteinSym}, where it is proven that, under the additional assumption that 
$n$ can be written in the form $2^{m_1}+\cdots+2^{m_r}$ for some $m_1\ge \cdots m_r\ge 0$, 
$$\ed_k(\frakS_n) \le n-4.$$
This applies to the cases $(n,p)= (5,5), (6,3)$, and gives the equalities for the Cremona and essential dimensions in these cases.
\end{example}

If it is true that $\cre_k(\bbZ/p^n\bbZ) = n$, then  
$$\cre_k(\frakS_n) \ge \Big[\log_pn\Big], \quad \ed_k(\frakS_n) \ge \Big[\log_pn\Big]. $$ 

\vskip5pt
In a positive characteristic $p > 0$, a simple algebraic group of exceptional type may contain  
large finite simple groups \cite{Serre1}. For example, if $p = 7$, the adjoint group of type $E_6$ contains the Mathieu group 
$M_{22}$. We will see in the next section that $\cre_k(E_6) = 16$. This gives $\cre_k(M_{22}) \le 16$. 
Another example is the Janko groups $J_1$ and $J_2$, which are contained 
in a simple group of type $G_2$ in characteristic $p = 2$ and $p = 11$. We will show that $\cre_k(G_2) = 5$. This gives 
$\cre_k(J_1)\le 5$ and $\cre_k(J_2)\le 5$.

The classification of wild pseudo-reflection groups is also known (see \cite{Kemper} and references therein). 
Among them  are the reduction modulo  $3, 5$ or $7$  of the pseudo-reflection groups in characteristic zero with 
Nos. 23 -- 37, except groups Nos. 25, 26, 27. Note that not all of them have the polynomial ring of invariants 
(for example, the group $F_4$ with  No. 28). However, the field of invariants is always rational. One of 
the examples is the group 
$G = 3.\frakA_7\times2$ of degree $3$ in characteristic $5$. It is known that the group 
$\frakA_7$ is contained in $\PGL_3(k)$, so $\cre_k(\frakA_7)\le 2.$ The realization of $3.\frakA_7$ as a 
pseudo-reflection group 
in characteristic $p = 5$, shows that
$$\cre_k(3.\frakA_7\times 2) = 3$$
(it is not difficult to show that $\cre_k(3.\frakA_7) > 2$). Note that the reductions of the Weyl groups 
$W(E_6)$ (resp. $ W(E_7), W(E_8))$ are wild reflection groups in characteristics $5$ (resp. $3,5,7$).
I do not know the essential dimension 
of wild pseudo-reflection groups, except in a few cases. Unfortunately, the results of 
\cite{Bronson} and \cite{DuncanReichstein} do not apply to wild pseudo-reflection groups.

\section{Algebraic groups}
The notion of the essential dimension can be extended to any algebraic group. It is a special 
case of the essential dimension of any covariant functor 
$$\calF: \textrm{Fields}/k\to \textrm{Sets}$$
introduced by Merkurjev \cite{Merkuriev}. For any $L/k$ and $a\in \calF(L)$, we say that $\ed_k(a) =n$ if 
$n$ is the minimum of the transcendence degrees of subextensions $k \subset K \subset L$ 
such that $a$ is in the image of $\calF(K)\to \calF(L)$. The essential dimension $\ed_k(\calF)$ 
is the supremum of $\ed_k(a)$ for all $(L,a\in \calF(L))$. 

For example, taking the functor $X:K\mapsto X(K)$ of  points of an irreducible algebraic variety, we obtain 
that $\ed_k(X) = \dim_k(X)$. 

 The essential dimension  of an algebraic group $G$ is the essential dimension of the functor 
$K\to H^1(K,G)$, the set of $G$-torsors of $G$ over $K$.  
 Connected linear algebraic groups of essential dimension zero were classified by Grothendieck. Serre introduced such 
 groups earlier and referred to 
 them as \emph{special groups}. An algebraic group is special if its maximal connected semisimple subgroup is the product 
 of groups $\SL_n(k)$ or $\Sp_n(k)$.

Let $G$ be an algebraic group that acts rationally on an irreducible algebraic variety $X$. By a theorem of Rosenlicht, 
we can find an open subset $U\subset X$ such that the geometric quotient $U/G = Y$ exists.  
Suppose also that we can find $U$ such that $G$ acts freely on it.
Then, the projection $\pi:U\to Y$ is a principal $G$-bundle, or, in other terms, a torsor 
over $Y$ \cite[Proposition 0.9]{Mumford}.
 By the definition of the essential dimension,
$$\ed_k(G)\le \dim Y = \dim X-\dim G.$$ 
In fact, 
$$\ed_k(G) = d-\dim G,$$
where $d$ is the smallest dimension of the image of a rational map $\phi:Y\to Z$ such that there exists a commutative 
diagrm 
$$\xymatrix{U\ar[d]^{\pi}\ar[r]& V\ar[d]^{\pi'}\\
Y\ar[r]^\phi&Z,}
$$
where $\pi':V\to Z$ is a $G$-torsor. 
This makes clear why $\ed(\GL_n(k)) = 0$: one should consider the linear action of the group on the linear space of matrices by left multiplication. 

In the following, we assume that $G$ is a simple algebraic group and $\cha(k) = 0$. We have the isogenies
$G^{\textrm{sc}}\to G\to G^{\textrm{adj}}$, where $G^{\textrm{sc}}$ is simply connected and $G^{\textrm{adj}}$ is adjoint. The \emph{rank} of $G$ is the dimension of its maximal torus.

 The existence of a linear representation $\rho:G\to \GL(V)$ (as always assumed to be a rational representation) 
such that $G$ acts generically freely on $V$, i.e., acts freely on an open subset of $V$, gives an upper bound for 
$\ed_k(G)$. 

Recall that a finite abelian subgroup of an algebraic group is called \emph{toral} if it is contained in a maximal torus of the algebraic group. 
One can obtain a lower bound for $\ed_k(G)$ by using that, for any finite abelian non-toral subgroup $A$ of $G$ (i.e., 
not contained 
in a maximal torus of $G$),
\beq
\rank(A)\le \ed_k(G)
\eeq
\cite[Theorem 7.2]{ReichsteinYoussin}. An abelian toral subgroup $A$ has rank $\le \rank(G)$. A non-toral abelian 
subgroup may have a larger rank. 
It is known that, for any abelian non-toral $p$-group $A$ in a simply connected simple algebraic group, $p$ divides the order of the 
Weyl group of $G$ \cite[1.2.2]{Serre1}. Any subgroup of the fundamental group of $G$ is non-toral. 
Any simply connected $G$, 
not of types $A_n$ or $C_n$, contains abelian non-toral subgroups. A simple algebraic group $G$ satisfying $
G^{\adj} = G^{\textrm{sc}}$ must be of types $G_2,F_4$ or $E_8$. These groups  
 contain non-toral $2$-groups $2^3, 2^5, 2^9$, respectively \cite{Griess}. This gives a bound
$$\ed_k(G_2)\ge 3, \quad \ed_k(F_4) \ge 5, \quad \ed_k(E_8) \ge 9.$$
On the other hand, the group $E_6$ (adjoint or simply connected) contains a group $2^5$
(coming from a subgroup isomorphic to the exceptional group $F_4$).  
The group $E_7^{\adj}$ (resp. $E_7^{\sc}$ contains $2^8$ (resp. $2^7$). This gives
$$\ed_k(E_6) \ge 5, \quad \ed_k(E_7^{\adj}) \ge 8, \quad \ed_k(E_7^{\textrm{sc}}) \ge 7.$$

\vskip6pt
 Let us now discuss the Cremona dimension $\cre_k(G)$ 
of an affine algebraic $k$-group. It can be defined in a similar manner to that for a finite group $G$. 
However, my conjecture does not extend to this 
case. For example, the essential dimension of a special algebraic group is zero, but the Cremona dimension 
is obviously positive. 

Another, more interesting example, we use that $\ed_k(\PGL_n(k)) = 2$ for $n = 2, 3 , 6$ \cite{ReichsteinInf}.
 However, applying \eqref{demazure}, we see that
$\PGL_{n+1}(k)$ does not embed in $\Cr_\bbC(m)$ for any $m <  n$.  On the positive side, 
$\ed_k(\textrm{PO}_{n+1}(k))) = n$, and it is known that 
$\textrm{PO}_{n+1}(k)$ embeds in $\Cr_k(n-1)$ via a birational map from $\bbP^{n-1}$ to a quadric in $\bbP^{n}$.


A maximal torus $T$ of $G$ contains a finite abelian group $p^r$, where $p$ is any prime number and $r = 
\dim(T)$ is the rank of $G$.  Applying 
\eqref{kollar}, we obtain
\beq\label{demazure}
\rank(G) \le \cre_k(G).
\eeq
Connected algebraic groups of rank $n$ contained in $\Cr_k(n)$ were classified by Demazure \cite{Demazure}. 
His work extends Enriques' work in the case $n = 2$.

A well-known theorem of Weil asserts that a birational action of an algebraic group can be regularized 
(see the references in \cite{Blanc2}). This means that 
an algebraic subgroup $G$ of $\Cr_k(n)$ acts biregularly on a rational variety $X$. The group of automorphisms $\Aut(X)$
of a complex rational variety (or even rationally connected variety) is a linear algebraic group of dimension equal 
to the dimension of the linear space of regular vector fields. So, if $G$ embeds in $\Cr_\bbC(n)$, there exists 
a rational smooth algebraic variety $X$ of dimension $n$ such that $h^0(X,T_X) \ge \dim G$. 
Many such varieties are found among projective vector bundles. The classification of connected 
algebraic subgroups of $\Cr_\bbC(2)$ and $\Cr_\bbC(3)$ is known (see \cite{Blanc} and the references thereein).

A connected algebraic subgroup of $\Cr_k(n)$ is said to be of \emph{maximal rank} if it contains a torus 
$T$ of dimension $n$. 
It is known that all such tori are conjugate 
in $\Cr_k(n)$. The classification of semi-simple connected algebraic subgroups of $\Cr_K(n)$ that contain 
a maximal torus defined over a field $K$ is known \cite[\S 4]{Demazure}. Each such subgroup is 
is a maximal algebraic subgroup of $\Cr_K(n)$ and is isomorphic to the connected component 
of the identity of the automorphism group of the product of the projective spaces 
\beq
 (\bbP_K^{m_1})^{n_1}\times \cdots \times (\bbP_K^{m_r})^{n_r}, \quad 
m_1n_1+\cdots+m_rn_r = n.
\eeq

 \begin{example}\label{grassmann}  In this example, we will compute $\cre_\bbC(G)$ of any adjoint simple algebraic linear group $G$ 
 over $k$ of characteristic $0$.\footnote{As there is no common agreement on the definitions, we assume that 
 $G$ is connected, and we admit almost simple groups (e.g., $\SL_n(k)$) as simple groups.}

 Suppose $G$ acts regularly on a smooth irreducible projective variety $X$ of dimension $d$. 
If $G$ has a fixed point $x\in X$, then $G$ acts faithfully on the tangent space $T_x(X)$, and, since $G$ is simple and adjoint, 
it acts faithfully on the projective space of dimension $d-1$. Replacing $X$ by $\bbP^{d-1}$, we may assume that 
$X^G = \emptyset$. Then, $G$ acts faithfully on each orbit in $X$. Replacing $X$ by a closed orbit of smallest dimension, we may assume that 
$X$ is a homogeneous space over $G$. After resolving singularities of $X$, we may assume that $X$ is a smooth projective 
homogeneous space, a quotient $G/P$ by a parabolic subgroup $P$. Since we want to minimize the dimension of $X$, we may assume that 
$P$ is a maximal parabolic, so 
$$\cre_\bbC(G) = M:=\min_{\textrm{$P$ maximal parabolic}} \dim G/P.$$

 It is known that a maximal parabolic subgroup $P$ is determined by a simple root of its root system, 
 in other words, a node in its Dynkin diagram.
 The dimension of $G/P$ is equal to the number of positive roots, which are not supported in the complementary  
 set of nodes.
 
 Easy computations show that 
 $$M = l,\  2l-1,\  2l-1,\ 2l-2, \  16, \ 27, \ 57, \ 15, \ 5,$$
 if $G$ is of type $A_l,  B_l , C_l,  D_l (l\ge 4),  E_6,  E_7,  E_8,  F_4,  G_2$, respectively. 
 The classical groups corresponding to $A, B, C, D$ series are
 $\mathrm{PGL}_{l+1}, \mathrm{PO}_{2l+1} , \PSp_{2l}$, and $\mathrm{PO}_{2l}$ (we also use that 
 $\mathrm{PSO}_4\cong \mathrm{PSO}_5$ and $\mathrm{PSO}_6\cong \PSp_4$).  
 The homogeneous varieties corresponding to classical groups are projective spaces, or
  quadrics. The homogeneous varieties $X$ corresponding 
 to the exceptional groups are the $E_6$-variety in $\bbP^{26}$ equal to the 
 singular locus of the Cartan cubic hypersurface, the $E_7$-variety in $\bbP^{55}$ equal 
 to the singular locus of the subvariety $\textrm{Sing}(X)$ of singular points of the Cartan quartic hypersurface \cite{Kempf}, 
 the  minimal adjoint orbit of the group $E_8$ in $\bbP^{247}$, the minimal adjoint orbit of $F_4$ in $\bbP^{51}$ or 
 a linear section of the $E_6$-variety by a $F_4$-invariant hyperplane, or 
 the minimal adjoint orbit of $G_2$ in $\bbP^{13}$. It is a quartic hypersurface in  
 a linear section 
 of the $E_6$-variety by a $G_2$-invariant 
 subspace $\bbP^6$ of $\bbP^{26}$ (see, 
 for example, \cite[\S 6]{Landsberg}). 
 
For example, besides the products of projective general groups, 
$\Cr_\bbC(2)$ contains none of the other simple algebraic groups. However, $\Cr_\bbC(3)$ contains $\PO_5\cong \PSp_4$. 
The group 
$\Cr_\bbC(4)$ contains $\PO_6$. The first case, 
where a simple algebraic group $G$ of one of the exceptional types appears, is $n = 5$ and $G \cong G_2$.

 \end{example}
 
\begin{example}
We can use \eqref{obvious} to get an upper bound for the Cremona dimension of a finite simple group. For example, 
it is known that $L_2(13)$ is contained in a simple group $G_2$. This gives 
$$4\le \cre_\bbC(L_2(13)) \le 5.$$ 
The minimal dimension of an irreducible representation of $L_2(13)$ is equal to $7$. 
Another example is the group $L_3(8)$, which is contained in a simple adjoint group $E_6$. We have
$$\cre_\bbC(L_3(8))\le 16.$$
It is known that the minimal dimension of an irreducible representation of this group is equal to $72$.
\end{example}

\section{Lattices in simple Lie groups}
\vskip6pt
Although the essential dimension of a discrete subgroup of a Lie group is not defined, 
I will briefly discuss the Cremona dimension of such a group. 

Let $\Gamma$ be a lattice in a simple real Lie group $G$ (in our definition, this means that $G$ is 
connected and its Lie algebra is 
simple). 
Let $d(\Gamma)$ be the smallest dimension of a faithful linear representation of $\Gamma$. It is known to be 
strictly larger than $\rank_\bbR(G)$ if $\rank_\bbR(G) \ge 2$. Suppose that $\Gamma$ embeds in the 
group $\Bir(X)$ of birational automorphisms 
of a complex algebraic variety $X$. It is conjectured that 
$$d(\Gamma) \le \dim X-1.$$
In particular, if $\rank_\bbR(G) \ge \dim X$, then $\Gamma$ is not isomorphic to a subgroup of $\Bir(X)$.
This is known to be true for non-cocompact lattices. 
Moreover, if the equality takes place, then $G$ is isogenous to $\SL_{\dim X+1}(\bbR)$ \cite{CantatXie}.

For example, $\Gamma = \SL_{n+1}(\bbZ)\subset \SL_{n+1}(\bbR)$ embeds in $\PGL_{n+1}(\bbC)$ for $n$ odd 
and, hence, embeds in $\Cr_\bbC(n)$. For even $n$, it can be embedded in $\Cr_\bbC(n+1)$ 
as a subgroup of monomial birational transformations of $\mathbb{A}_\bbC^{n+1}$. However, $\Gamma$ cannot be 
embedded in $\Cr_{\bbC}(n-1)$ for even or odd $n$.

The group $\Or(1,n)$ of orthogonal transformations of the hyperbolic space $\bbR^{1,n}$ is of the real rank $1$. 
It is known that for some $n$ (e.g. $n = 9$), it contains non-cocompact lattices embedded in $\Cr_\bbC(2)$.
A notorious example is the group of automorphisms of a general rational Coble surface isomorphic to such a lattice 
in $\Or(1,9)$.

\end{document}